\documentclass[11pt]{amsart}
\usepackage{geometry}                
\usepackage{subfig}
\usepackage{graphicx}
\usepackage{amssymb}
\usepackage{epstopdf}
\usepackage{bbm}
\usepackage[foot]{amsaddr}

\newtheorem{theorem}{Theorem}
\newtheorem{lemma}[theorem]{Lemma}
\newtheorem{corollary}[theorem]{Corollary}

\theoremstyle{remark}
\newtheorem{remark}{Remark}[theorem]

\newcommand{\R}{\mathbb{R}}

\renewcommand{\L}{\mathbf{L}}
\renewcommand{\l}{\mathbf{l}}
\renewcommand{\O}{\mathbf{\Omega}} 

\newcommand{\A}{\mathbf{A}}

\renewcommand{\v}{\mathbf{v}}

\newcommand{\dee}{\mathrm{d}}

\newcommand{\rem}[1]{}
\newcommand{\scaledT}{{\hat T}}
\newcommand{\scaledP}{{\hat P}}

\title{The Diver with a Rotor}
\author[S.~Bharadwaj]{Sudarsh Bharadwaj$^1$}
\author[N.~Duignan, H.R.~Dullin]{ Nathan Duignan$^2$, Holger R.~Dullin$^2$}
\author[K.~Leung]{Karen Leung$^1$}
\author[W.~Tong]{William Tong$^2$}

\address{$^1$ Australian Centre for Field Robotics, The University of Sydney}
\address{$^2$ School of Mathematics and Statistics, The University of Sydney}
\email{holger.dullin@sydney.edu.au}

\begin{document}
\maketitle

\begin{abstract}
We present and analyse a simple model for the twisting somersault. 
The model is a rigid body with a rotor attached which can be switched on and off.
This makes it simple enough to devise explicit analytical formulas whilst still maintaining sufficient complexity to preserve the shape-changing dynamics essential for twisting somersaults in springboard and platform diving.
With ``rotor on" and with ``rotor off" the corresponding 
Euler-type equations can be solved, and the essential quantities characterising 
the dynamics, such as the periods and rotation numbers, can be computed in terms
of complete elliptic integrals. 
Thus we arrive at  explicit formulas for how to achieve a dive with 
$m$ somersaults and $n$ twists in a given total time. 
This can be thought of as a special case of a geometric phase formula due to Cabrera \cite{Cabrera07}.
\end{abstract}

\section{Introduction}

The analysis of the twisting somersault poses an interesting problem in classical mechanics. 
How can a body take off in pure somersaulting motion, 
initiate twisting in midflight,
and then return to pure somersaulting for entry into the water?
Generally this is not a problem of rigid body dynamics, but instead of 
either non-rigid body dynamics or the description of coupled rigid bodies. 
Such a description of the twisting somersault was first proposed by 
\cite{Frohlich79} and has then been developed into a  full fledged analysis 
by Yeadon in a series of classical papers \cite{Yeadon90b,Yeadon90d,Yeadon93a,Yeadon93c,Yeadon93d}.
Here we are less ambitious in that we  develop  possibly the simplest model
capable of exhibiting this kind of behaviour.
The advantage of our model is that it is simple enough it can be completely solved, 
and thus we are able to derive an exact equation that determines how exactly 
$m$ somersaults and $n$ twists can be performed in the total time $T_{tot}$, 
if at all. The model consists of a rigid body with a rotor attached, and the 
rotation can be turned on and off. The question we can answer is this: ``When 
does the rotor need to be turned on to initiate twisting, how long should it stay on, 
off, and then on again to stop the twisting". 
From the dynamical systems point of view we have two autonomous systems
(``rotor on" or ``rotor off") which are switched between at times to be determined 
to achieve the desired trajectory.
As such it is a discontinuous dynamical system, whose solution is at least continuous.
Despite its simplicity, the model appears to capture the essential 
features and even reasonable values of the parameters that are relevant in human
springboard and platform diving. 
Whether we can learn something about human diving from this model --
other than a rough idea of the fundamental principles -- remains to be seen.
However, we would like to propose that the simple device we are describing would 
make  an interesting robot capable of performing twisting somersaults, 
potentially with many more twists than humanly possible.

\begin{figure}
\includegraphics[width=2cm]{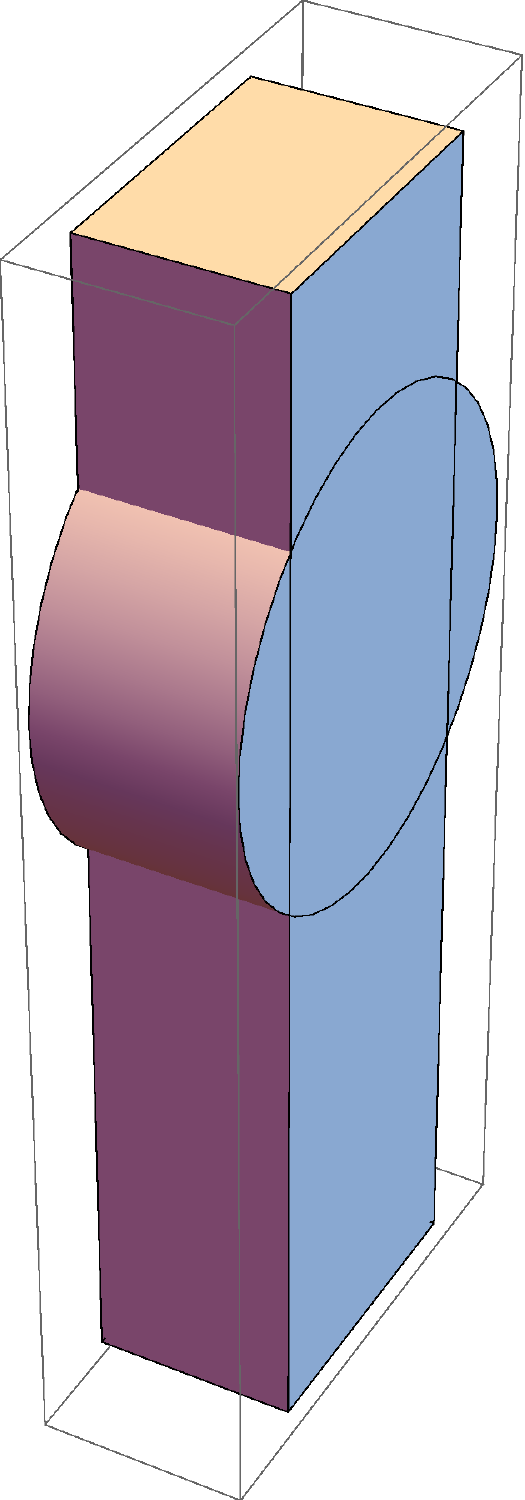}
\caption{A possible model of a rigid body (the box) and a disc that can be made 
to rotate about its symmetry axis. The box models the head, trunk, and legs of 
a human diver, while the disc models the arms, which can ``rotate''. In a robotic
realisation the disc would actually rotate.} \label{fig:model}
\end{figure}

\section{Euler equations for a rigid body with a rotor}

Let $\l$ be the constant angular momentum in a space fixed frame,
and $\L$ the angular momentum vector in a reference frame moving 
with the body. Let $R$ be the rotation matrix that transforms 
from one frame into the other, so that $\l = R \L$. 
The equations of motion for a rigid body with a rotor attached are well known, 
see, e.g., \cite{Wittenburg77,Leimanis65,CohenMuncaster88,Koiller84}.
Following Yeadon \cite{Yeadon93a} we use an adapted system of Euler angles 
$R =  R_1(\phi)R_2(\theta) R_3(\psi)$ where $R_i$ is a rotation that fixes 
the $i$th axis, $\phi$ is the somersault angles, $\theta$ the tilt angle, 
and $\psi$ the twist angle. 
This is the Euler-angle convention typically used in aerospace engineering 
where the angles are referred to as roll, pitch, and yaw. 

\begin{theorem}
The equations of motion for a rigid body with a rotating disc attached are given by 
\[
     l \begin{pmatrix} \cos\theta \cos\psi \\ -\cos\theta \sin\psi \\ \sin\theta \end{pmatrix} 
     - \begin{pmatrix} 0 \\ h \\ 0 \end{pmatrix}
      = \begin{pmatrix} I_1 & 0 & 0 \\ 0 & I_2 & 0 \\ 0 & 0 & I_3 \end{pmatrix}
      \begin{pmatrix} \cos\theta \cos\psi & \sin\psi & 0 \\ 
                 - \cos \theta \sin \psi & \cos\psi & 0 \\
                 \sin\theta & 0 & 1 \end{pmatrix}
      \begin{pmatrix} \dot \phi \\ \dot \theta \\ \dot \psi \end{pmatrix} \,.
\]
\end{theorem}
\begin{proof}
We start with the general Euler equations as e.g.\ derived in \cite{Tong15},
\[
   \dot \L = \L \times \O, \quad \O = I^{-1} ( \L - \A)\,,
\]
where $\L$ is the angular momentum in a body frame, 
$I$ the tensor of inertia, and $\A$ the internal angular momentum 
generated by the rotating disc. 
Using the constancy of $\l$ we can write the equations of motion as
\[
     R^{t} \l = I \O + \A
\]
where  $\O$ is determined by $R$ through
$\O \times \v =  R^t \dot R \v$ for any vector $\v \in \R^3$.

By choice of the space fixed coordinate system we may assume that $\l = ( l, 0, 0)^t$.
This means that the space fixed coordinate system is oriented so that the 1-axis is 
pointing to the right, the 2-axis to the front, and the 3-axis upwards.
With the above choice of Euler angles we can find $\O$ as on the right hand side of the formula
and similarly $\L = R^t \l$ on the left hand side of the formula.
 \end{proof}

When the rotor is off we have $h = 0$ and hence the internal momentum $\A$ vanishes.
In this case the classical Euler equations are recovered.
When the rotor is on we have $h$ and hence $\A$ is non-zero 
but constant, and the equations of motion are as given in the theorem.
 
Due to the circular symmetry of the disk the moment of inertia tensor $I$ of the ``diver'' 
will be constant whether the disc is rotating or not. This is the essential 
simplification which makes this model tractable as compared to a general shape change
which induces a time-dependent angular momentum shift 
{\em and} a time-dependent tensor of inertia $I$.

Consider the caricature of a diver by a rectangular box (the trunk with legs and 
head attached, all rigidly connected) with a disc attached 
as shown in Figure~\ref{fig:model}.
The reference configuration is such that the 1-axis is pointing to the side of the body,
the 2-axis out of the chest of the body, and the 3-axis up towards the head. 
The disc is attached so that it can rotate about an axis through the chest. 
The idea is to use the rotating disc to model the rotational up / down motion 
of the arms (and the hip and legs to a lesser extent).
This will generate an internal angular momentum about the 2-axis,
 so that when the disc is rotating we have $\A = (0, h, 0)$.

We have $h = \omega_d I_d$ where $\omega_d$
is the angular velocity of the disc, and $I_d$ is its moment of inertia for rotation
about its symmetry axis. These parameters need to be chosen so that we 
have a rough correspondence to the arm throw that initiates and stops the 
twisting motion. We estimate that moving the arm from ``up" to ``down", 
i.e.{} through an angle $\pi$ takes at least $0.25$ seconds, so that $\omega_d \le 4\pi$.
Modelling each arm by solid cylinders gives a value of $I_d \approx 2$, roughly one for each arm.
It seems plausible to think of the disk as modelling the simultaneous motion of both 
arms rotating in the same direction, remaining parallel.
Both these numbers are just ballpark figures. As compared to the moments of 
inertia for the whole body, which we take to be $(I_1, I_2, I_3) = (20,21,1)$ (with both arms up),
the rule of thumb is 
that the moment of inertia of the disk is comparable to the moment of inertia
for pure twisting of the whole body. These figures are computed from the 
model used in \cite{Tong15}.
Much of our detailed analysis is done for the symmetric case $(20,20,1)$.
In the general case we only consider $I_1 < I_2$, so that the initial somersaulting 
takes place about the unstable middle axis of the body.
Another ballpark figure to keep in mind is $l \approx 2\pi  \cdot 20$ for the angular momentum 
of the whole body, which corresponds to the angular momentum necessary to 
perform one full somersault in one second.

\begin{corollary} \label{cor:scale}
The equations of motion can be written as
\[
\begin{aligned}
     \phi' & = 
    			1 + \delta \sin^2 \psi                 + \hat \rho \sec\theta  \sin\psi \\
    \theta' & =  -\delta \cos\theta \cos\psi\sin\psi - \hat \rho  \cos\psi \\
    \psi' & = 
                  \gamma \sin\theta - \delta \sin\theta \sin^2 \psi  - \hat \rho \tan\theta \sin\psi
\end{aligned}
\]
where the prime $'$ denotes the derivative $d/d\tau$ with respect to the dimensionless time $\tau = t l/I_1 $ and
\[
  \delta =  \frac{I_1}{I_2} - 1, \quad 
  \gamma  = \frac{I_1}{I_3} - 1,  \quad 
  \rho = \frac{h}{l}, \quad
  \hat \rho = \rho ( 1 + \delta) 
\]
are dimensionless parameters.
The dimensionless energy is a constant of motion and is given by 
\[
       E = \frac12 \left( 1 + \gamma \sin^2\theta  + \delta \cos^2 \theta \sin^2\psi \right) 
       + \frac12 \hat \rho ( \rho + 2 \cos \theta \sin \psi) \,.
\]
\end{corollary}
\begin{proof}
Taking the equations from Theorem 1 and dividing by $l$, then scaling time 
by $I_1/l$ non-dimensionalises the equations, and solving for the derivatives of the angles
gives the equations with the dimensionless parameters $\sigma$, $\delta$, and $\rho$ as stated.
The energy is given by $E = \tfrac12 (\L - \A)^t I^{-1} (\L - \A)$ where $\L = R^t \l$.
The fact that it is a constant of motion can be shown by direct computation,
or alternatively using the fact that it is the Hamiltonian of the flow with respect to the Poisson 
structure $\L \times$. 
Non-dimensionalisation and expressing this in Euler angles gives the result.
\end{proof}

\begin{remark}
The symmetric case is found for $\delta = 0$ and the case with the rotor fixed is $\rho = 0$.
In all cases $\phi$ does not appear on the right hand side. 
\end{remark}

\begin{remark}
For $\rho = 0$ the pure somersaulting equilibrium (ignoring $\phi$) is at
$\theta = \psi = 0$, and the eigenvalues of the linearisation about this equilibrium 
are $\pm \sqrt{ -\delta \gamma}$. In particular the somersault is unstable 
when $\delta < 0$. This corresponds to $I_2 < I_1 < I_3$, which will be our 
standard assumption in the general case.
\end{remark}

\begin{remark}
The non-dimensionalisation measures time in units of the inverse angular frequency $l / I_1$ 
of the pure somersault. Hence in scaled time after time $2\pi$, a full pure somersault 
is executed, corresponding to $\phi' = 1$ when $\rho = \delta = 0$. 
This is why in the following figures we often plot the period divided by $ 2 \pi$.
\end{remark}

\section{The symmetric case}

The dive can be separated into rigid and non-rigid stages. 
In the rigid stage $\rho = 0$ and for a symmetric body $\delta = 0$.
Therefore, the equations of motion become trivial with 
$ \phi' = 1$, $ \theta' = 0$, and $ \psi' = \gamma \sin\theta$ (a constant).
Even when $\rho \not = 0$ (``rotor on'') the equations of motion 
can be expressed completely in terms of $\theta$:

\begin{corollary} \label{cor:symmetric}
In the symmetric case $I_1 = I_2 \Longleftrightarrow \delta = 0$;  the equations of motion  are 
\[
  \begin{aligned}
       \phi'  & =  1 + \rho \sec \theta \sin \psi  
      \\ 
       \theta' & = -\rho \cos\psi 
      \\
       \psi' & = \gamma \sin \theta - \rho \tan \theta \sin\psi 
  \end{aligned}
\]
where 
\[
    \sin\psi = \frac{E - \frac12 ( 1 + \rho^2) }{\rho \cos\theta} - \frac{\gamma}{2 \rho} \sin\theta \tan\theta\,.
\]
\end{corollary}
\begin{proof}
Solve the energy equation with $\delta = 0$ for $\sin\psi$ gives the result.
Thus (at the expense of a square root in the $\dot \theta$ equation)
the angle $\psi$ can be eliminated on the right hand side.
\end{proof}


The dive starts with a rigid stage (``rotor off'') in pure somersaulting where $\L = \l = (l, 0, 0) = const$.
In this stage we have $ \phi' =1$ and $\theta  = \psi = 0 = const.$
The time in stage $i$ is denoted by $T_i$ or $\scaledT_i$ for the dimensionless time, 
the amount of somersault by $\phi_i$ and 
the amount of twist by $\psi_i$. In stage 1 (pure somersault) we have $\psi_1 = 0$,
and $\phi_1 = \scaledT_1$.

In the non-rigid stage 2 the rotor is switched on. 
When the rotor is switched on the trajectory  starts at $\L = (l,0,0)$ and we let it run
until it reaches the maximum possible value of tilt $\theta_{max}$ along that orbit,
because then the twist in the next stage will be fastest, since $ \psi' = \gamma \sin \theta$.
In the next two lemmas we are going to compute the amount of time $T_2$ this takes,
and the amount of somersaulting $\phi_2$ that occurs during this time.
Let us remark that $\phi_2$ can be interpreted as one quarter of the rotation number 
of the integrable system rigid body with a rotor. 

\begin{figure}
\includegraphics[width=10cm]{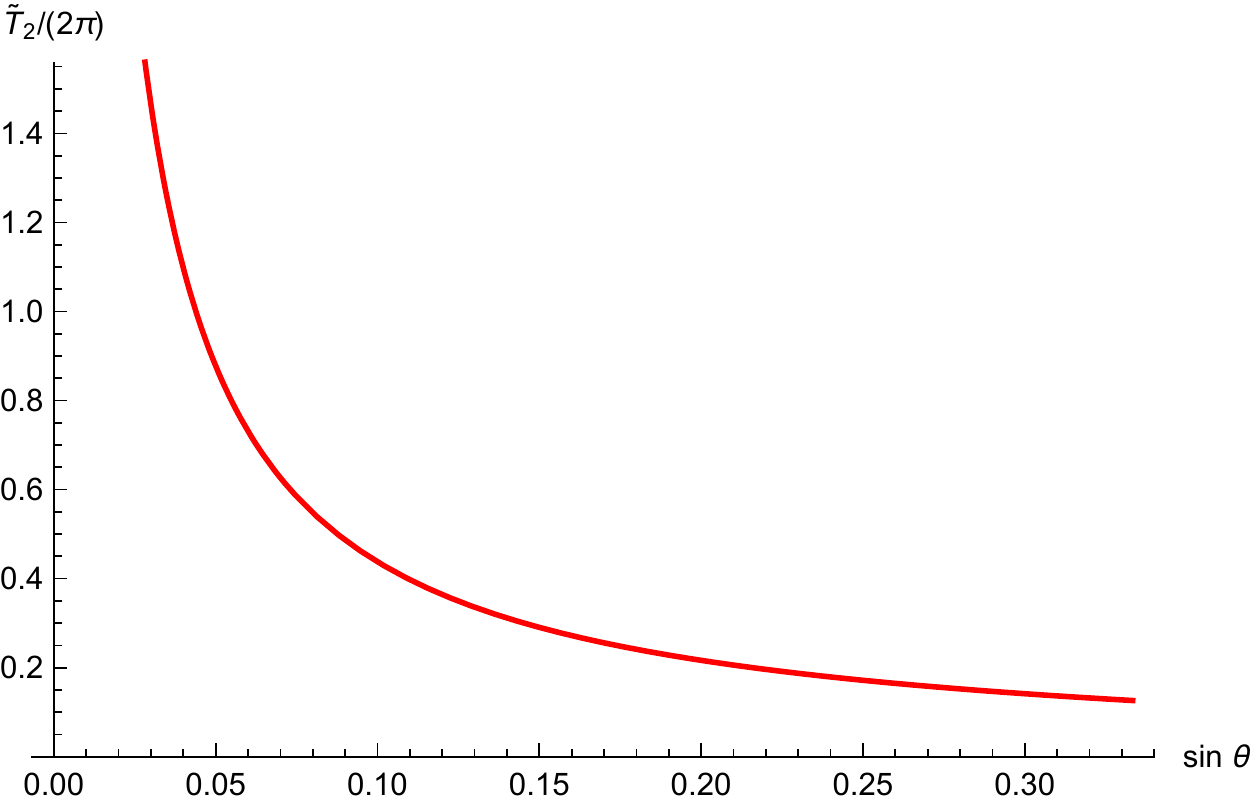}
\caption{Scaled time with rotor on $\scaledT_2 = l T_2 / I_1$ as a function of maximal tilt $s=\sin \theta_{max}$, $\gamma = 19$, as given by \eqref{eqn:T2}.} \label{fig:t2plot}
\end{figure}

\begin{lemma} \label{lem:T2}
The maximal $\theta$ that can be reached with ``rotor on'' from $ \L = (l, 0, 0)$ is given by
\[
  \cos\theta_{max} = \sqrt{ \beta^2 + 1} - \beta, \qquad \beta = \frac{\rho}{\gamma} \,.
\]
The inverse relation is 
\[
       \beta = \frac{s^2}{2 \sqrt{ 1 - s^2}}, \quad s = \sin \theta_{max}.
\]
The time $T_2$ to move with ``rotor on'' from the point $\L = (l,0,0)$ to the point 
$\L = l(0,\cos\theta_{max}, \sin\theta_{max})$ is
\begin{equation} \label{eqn:T2}
    \frac{l}{I_1} T_2  = \scaledT_2  =   \frac{2 k }{s \gamma }  K( k^2 ), \quad
       k^2 = \frac{1 - s^2}{2 - s^2} \,.
\end{equation}
\end{lemma}
\begin{proof}
By discrete symmetry the maximum $\theta_{max}$ occurs for $\psi = \pm \pi/2$.
Then, from the energy equation in Corollary~\ref{cor:symmetric} we obtain 
$\sin^2 \theta \pm  2 \beta \cos \theta = 0$,
and hence the result. Using this equation $\beta$ can be eliminated
in favour of the variable $s = \sin \theta_{max}$.  

Considering the $ \theta'$ equation, separating the variables,
and integrating from 0 to $\theta_{max}$ gives 
\[
\int_{0}^{\theta_{max}} \frac{ -2 d \theta}{\gamma \sqrt{4 \beta^2 - \sin^2\theta \tan^2 \theta}} = \int d \tau = \scaledT_2 \,.
\]
This is a complete elliptic integral of the first kind, which can be put into algebraic form
with the substitution $z = \sin\theta$ so that the upper boundary is $s$. 
This can be expressed in terms of Legendre's $K$, see, e.g.~\cite{BF71}.
Un-scaling time gives the relation between $T_2$ and $\scaledT_2$.
\end{proof}

\begin{remark}
The scaled time  $\scaledT_2$ for stage 2 (up to the overall factor $1/\gamma$) 
depends on the maximal tilt angle $s = \sin\theta_{max}$ only,
 see Figure~\ref{fig:t2plot}.
We use the term ``maximal tilt angle'' for $s$ and $\theta_{max}$ interchangeably,
since they determine each other and for small tilt $s \approx \theta_{max}$.
\end{remark}

%

\begin{figure}
\includegraphics[width=10cm]{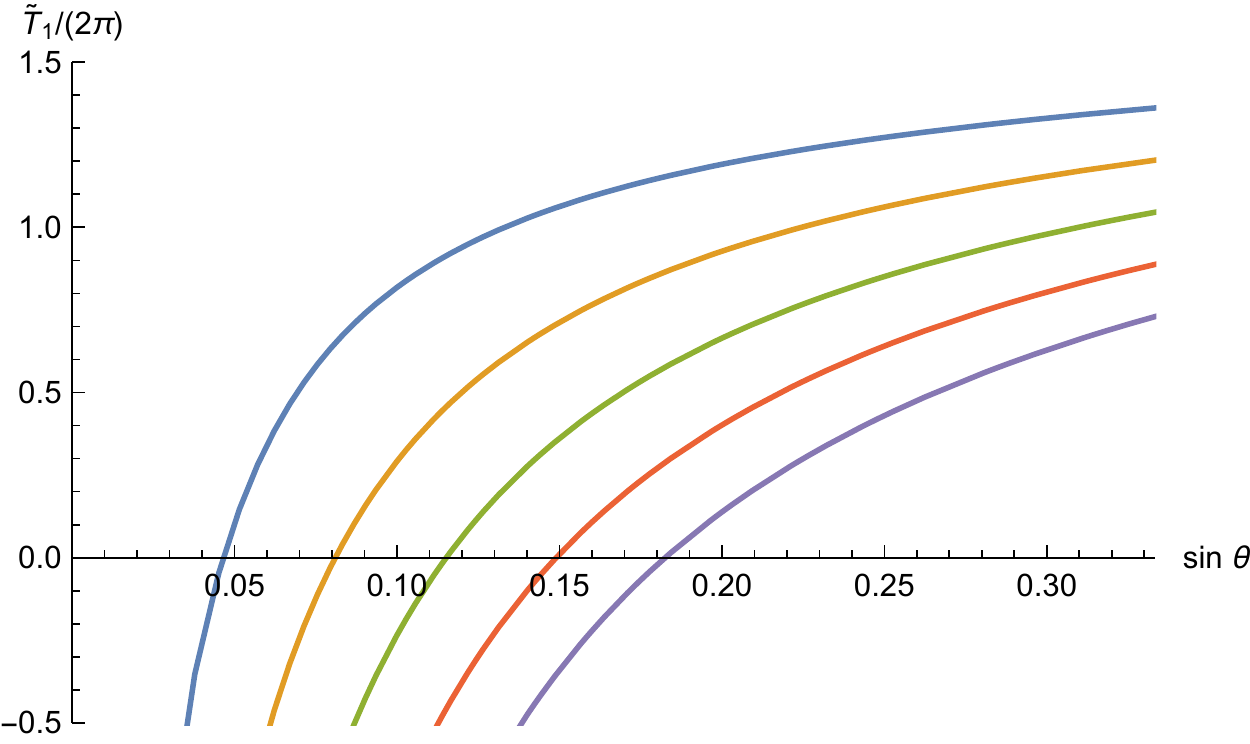}
\caption{Scaled somersaulting time $\scaledT_1 = l T_1 / I_1$ as a function of maximal tilt $s=\sin \theta_{max}$ for $n=1,2,3,4,5$ twists,
$m = 3/2$ somersaults, $\gamma = 19$, as given by \eqref{eqn:T1}. 
The dive for given $n$ is possible if $T_1 \ge 0 $.} \label{fig:t1plot}
\end{figure}

\begin{lemma} \label{lem:phi2}
The amount of somersault that occurs with ``rotor on'' is given by 
\[
       \phi_2 = \frac{1}{s} \left(  k( 1 + 2 \gamma^{-1}) K( k^2 )  - (k^{-1} + k) \Pi(2 - k^{-2} , k^2) \right)
\]
where  $k^2 = \frac{1 - s^2}{2 - s^2}$ as above.
\end{lemma}
\begin{proof}
Using the equation for $ \phi'$ from Corollary~\ref{cor:symmetric}, 
transformed as $\frac{d \phi }{d \tau} = \frac{ d \phi }{d \theta}  \frac{ d\theta }{ d\tau }$
with the equation for $ \theta '$ gives
\[
     \int_{0}^{\theta_{max}} \frac{2 - \gamma \tan^2\theta }{ \gamma \sqrt{4\beta^2 - \sin^2\theta \tan^2\theta} } d\theta  = \int d \phi = \phi_2 \,.
\]
This is a complete elliptic integral of the third kind which can be put into algebraic form
with the substitution $z = \sin\theta$ so that the upper boundary is $s$. 
Expressing it in terms of Legendre normal forms $K$ and $\Pi$, see, e.g.~\cite{BF71}, gives the result.
\end{proof}

\begin{remark}
$\phi_2$ is a function of the dimensionless parameters $s$ and $\gamma$ only.
The combination $\hat T_2 - \phi_2$ is a function of $s$ only.
\end{remark}

Now we have the basic ingredients to describe a full twisting somersault 
in the symmetric case $I_1 = I_2$.

\begin{figure}
\includegraphics[width=10cm]{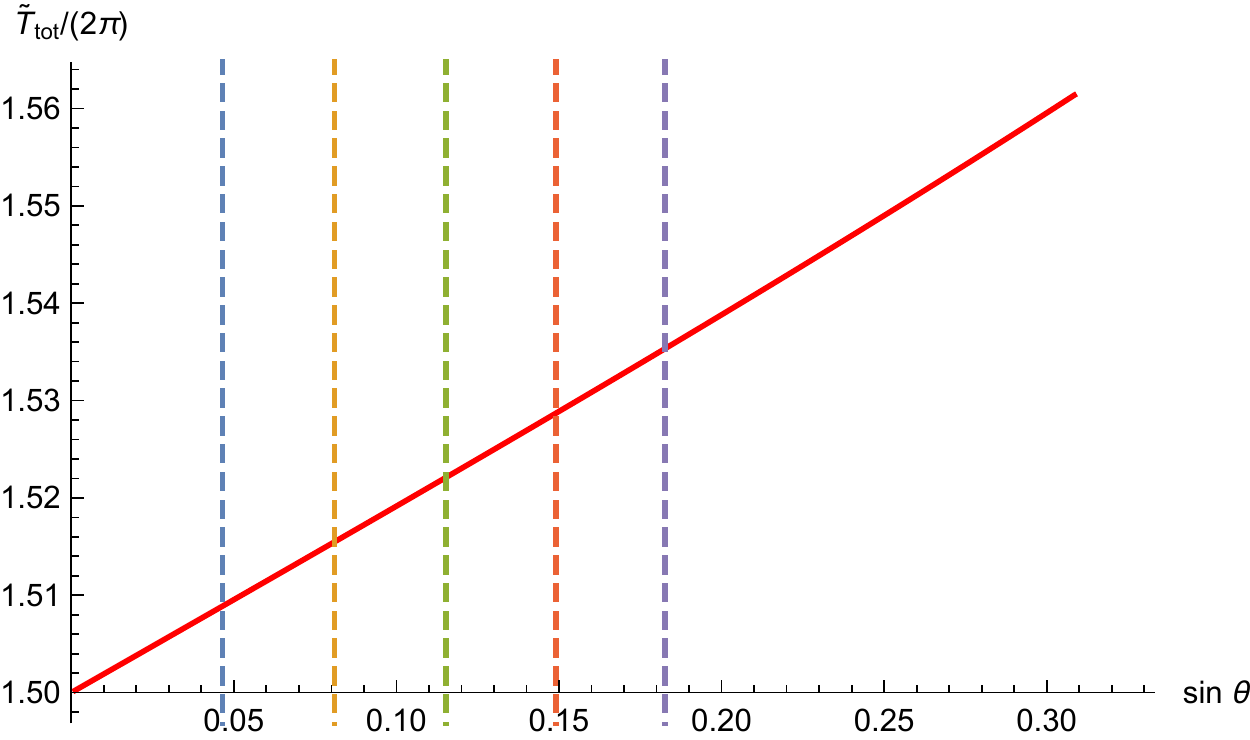}
\caption{Scaled total time $\scaledT_{tot} = l T_{tot}/I_1$ as a function of the maximal tilt $s=\sin \theta_{max}$ for  
$m = 3/2$ somersaults, as given by \eqref{eqn:Ttot}.
The dashed lines indicate the minimal tilt needed for $n=1,2,3,4,5$ twists.} \label{fig:lplot}
\end{figure}

\begin{theorem} \label{theorem:main}
To perform a dive with $m$ somersaults and $n$ twists 
in a total time $T_{tot}= 2 T_2 + T_3 + 2 T_1$ 
the master equation (see Figure~\ref{fig:lplot})
\begin{equation} \label{eqn:Ttot}
     \frac{ l}{I_1} T_{tot}  - 2 m \pi  =  2 \frac{l}{I_1} T_2 - 2 \phi_2
\end{equation}
needs to be satisfied,
where the right hand side depends on the maximal tilt $s = \sin\theta_{max}$ only.
In addition, the condition that $T_1$ is non-negative (see Figure~\ref{fig:t1plot})
\begin{equation} \label{eqn:T1}
     2 \frac{l}{I_1}  T_1 = 2 \pi m - 2 \phi_2 - \frac{ 2\pi ( n - \tfrac12)}{ \gamma \sin \theta_{max}}  \ge 0
\end{equation}
needs to be satisfied.
\end{theorem}
\begin{proof}
To achieve $m$ somersaults we need $ 2m\pi = 2 \phi_1 + 2 \phi_2 + \phi_3$.
To achieve $n$ twists we need $2 n \pi = 2 \psi_1 + 2 \psi_2 + \psi_3$.
In the symmetric case and with ``rotor off " (stage 1 and stage 3) 
the Euler equations simplify to $ \phi' = 1$, see Corollary~\ref{cor:symmetric},
so that $\phi_1 = \scaledT_1$ and $\phi_3 = \scaledT_3$.
The other Euler equations are $\psi' = \gamma\sin\theta $ and $\theta' = 0$.
In stage 1 we have $\theta = 0$, so $\psi_1 = 0$.
In stage 3 we have $\theta = \theta_{max} = const$, so $\psi_3 = \scaledT_3 \gamma \sin\theta_{max}$.
For later reference it is useful to introduce the scaled period of the twist $\hat P_3 = 2\pi / ( \gamma \sin\theta_{max})$.
The ``rotor-on'' stages 2 and 4 together produce a half-twist, $\psi_2 = \psi_4 = \pi/2$, so that
the condition to have $n$ twists gives
\[
  \frac{l}{I_1} T_3  = \scaledT_3  = (n-\tfrac12) \hat P_3 =  \frac{ 2 \pi ( n - \tfrac12)}{ \gamma\sin \theta_{max}}\,.
\]
For integer $n$ stages 2 and 4 are the same except for the sign of $\rho$.
For half-integer $n$ the same $\rho$ is used and the final somersaulting 
in stage 5 occurs with $\L = l(-1, 0, 0)$. In either case we set the contribution 
from stage 1 and stage 5 to be equal.

The condition to have $m$ somersaults gives
\[
  2 \frac{l}{I_1}  T_1 = 2 \scaledT_1  =  2 \pi m - 2 \phi_2 -\scaledT_3 \,.
\]
The times $\scaledT_2$ and $\scaledT_3$ are non-negative for $n \ge 1/2$ and are 
determined by $\gamma$ and the maximal tilt $s = \sin\theta_{max}$.
But for too large a value of $n$ the formula for $\scaledT_1$  gives a negative 
number, which means that for those values of $m$, $\gamma$, $\theta_{max}$ the dive with this $n$ is not possible.
 
Finally the total time is just the sum of the times of the stages.
Eliminating $T_1$ and $T_3$ gives the final formula 
$ \scaledT_{tot} - 2m \pi$ as a function of $s$ only, 
using the previous lemmas.
\end{proof}

%

When designing a jump one fixes $T_{tot}$, $n$, and $m$, 
and tries to find a solution to the equations for which $T_1$ 
is non-negative. 
Solutions appear  in one-parameter families, since the right hand side of 
\eqref{eqn:Ttot} just depends on $s$, while on the left hand side 
only  the product of $l$ and $T_{tot}$ is determined.
So in the space of parameters $h$ and $l$ a curve is defined.

However, if one imagines
that a particular solution is sought {\em after} takeoff, then $l$ and $T_{tot}$ are 
already fixed, and the only parameter still at  disposal 
is the speed of rotation $h$, and hence $s$. The fact that the solutions
appear in 1-parameter families in $(h,l)$ space,
see Figure~\ref{fig:lplot}, is thus crucial to allow for corrections after takeoff. 

%

The equations of the main theorem possess solutions for realistic values 
of the parameters (as discussed in the beginning, $\gamma \approx 19$, $T_{tot} \approx 1.5$, $h \le 8 \pi$, $l \le 50 \pi$)
for $m < 3$ and $n \le 4$.
Thus a typical value of $\beta = h/(l \gamma)$ is $\approx 0.01$, which gives a corresponding $s = \sin \theta_{max} \approx 0.14$,
and this is almost the tilt required to achieve $n=4$ twists, see Figure~\ref{fig:lplot}.
When $m$ is too big then by \eqref{eqn:Ttot} the necessary $l$ or $T_{tot}$ (or both) will be too big. 
When $n$ is too big then by \eqref{eqn:T1} the minimal necessary tilt 
$s = \sin \theta_{max}$ that will keep $T_1$ non-negative
will be so large that it cannot be achieved by a human diver.
A robotic model, however, could probably achieve the necessary values of $h$.
The reason we cannot do more somersaults with this simple model is that we
 have not included the possibility to go into pike or tuck position. This would 
 decrease $I_1$ and hence allow more somersaults for the same value of angular
 momentum $l$, since the essential parameter is the combination 
 $\scaledT_{tot} = l T_{tot} / I_1$.

Another interesting observation is that if $\scaledT_{tot}$ is sufficiently big,
which in practical terms means that $l$ is sufficiently big, (as $T_{tot}$ 
is essentially determined by the height of the platform) then after takeoff 
the diver can still decide how many twists to do by adjusting the tilt generated
and the timing of the individual stages.


To get a better understanding of what the two main equations 
\eqref{eqn:Ttot} and \eqref{eqn:T1} are saying and how  they 
depend on the parameters, we now give Taylor expansions valid for 
small maximal tilt angle $s$.

\begin{lemma} \label{lem:expand}
For small $s = \sin\theta_{max}$ we have the following leading order behaviour: 
\[
\begin{aligned}
    2 \scaledT_2 - 2\phi_2  &  \approx \sqrt{2} \left( 2 E(\tfrac12) - K(\tfrac12)\right) s + O(s^3) \\
    2 \scaledT_2                  &  \approx   \frac{2 \sqrt{2} }{s \gamma} K(\tfrac12) + O(s)
\end{aligned}
\]
\end{lemma}
\begin{proof}
Note that $T_{tot}$ does not depend on $\gamma$, 
specifically the right hand side of \eqref{eqn:Ttot} 
depends on $s$  (and hence $\beta = \rho/\gamma = h/(l \gamma)$) only:
\[
\scaledT_2  - \phi_2  =  \frac{1}{s} \left( -k K( k^2 ) + (k^{-1}  + k)\Pi(2 - k^{-2}, k^2 ) \right) \,.
\]
Furthermore, this combination is regular at $s = 0$, using $\Pi(0, k)  = K(k)$,
and since $k$ is an even function of $s$ the whole expression has vanishing limit for $s \to 0$.
The first $s$-derivative of the right hand side is 
$2(k^{-1} - k)(2 E - K)$ and hence the result.
By contrast, $\hat T_2$ has a pole at $s = 0$ and Taylor expansion of $s\hat T_2$ gives the result.
\end{proof}

\begin{theorem}
The leading order behaviour of the total scaled time $\scaledT_{tot} = l T_{tot} / I_1$ is determined by 
\[
    \frac{ \scaledT_{tot}}{2\pi} = m  + A s + O(s^3) \text{ where } A = \frac{ 1} { \sqrt{2}\pi}  \left( 2 E(\tfrac12) - K(\tfrac12) \right) \approx 0.1907 \,.
\]
The minimal  twist $\theta_{max}^*$ that is necessary to perform $m$ somersaults
and $n$ twist is approximately given by
\[
       \sin\theta_{max}^* \approx \theta_{max}^* \approx \frac{B + n}{m \gamma}, \text{ where } 
       B =  \frac{\sqrt{2}}{\pi} K(\tfrac12) - \frac12 \approx 0.3346 \,.
\]
\end{theorem}
\begin{proof}
Using Lemma~\ref{lem:expand} in the master equation for $\scaledT_{tot}$ gives the first result.
The minimal $\theta_{max}$ is determined by the condition that $T_1$ be equal to zero.
In the formula for $\scaledT_1$ one contribution comes from the pole in Lemma~\ref{lem:expand}, 
the other from the pole in $s$ in $\scaledT_3$, see Theorem~\ref{theorem:main}.
Combining them the equation $T_1 = 0$ can be approximately solved to get the stated result.
\end{proof}

This shows that (as expected) with more tilt more twists can be generated, and the additional
``cost'' in tilt of one twist is approximately $1/(m\gamma)$. 
It should be noted that increasing the tilt $s = \sin \theta_{max}$ also increases 
$\scaledT_{tot}$,  see Figure~\ref{fig:lplot},
and hence for fixed $T_{tot}$ it increases the necessary total angular momentum.
The minimal total scaled time that is feasible is simply given by the number 
of somersaults $m$, and increasing the tilt $s$ increases $\scaledT_{tot}$ by $sA$.

The agreement of these approximate formulas with the exact formulas shown in 
Figures~\ref{fig:t1plot} and \ref{fig:lplot} is very good in the relevant range of $s$.

\section{The general case}

When all three moments of inertia are distinct then even for the motion 
without rotor in the rigid twisting stage 3 the tilt $\theta$ is not constant,
and complete elliptic integrals are needed to express the time $T_3$ and
the somersault $\phi_3$. 
The corresponding expressions were elementary in the symmetric case.
Since $\theta$ is not constant $\sin\theta$ cannot be used as a parameter any more. 
Instead we will use $s_- = \sin\theta_{min}$ or $s_+ = \sin\theta_{max}$,
which are determined by the extremal values of $\theta$ in the twisting stage 3.
Thus from equating the energy in Corollary~\ref{cor:scale} for $\rho = 0$ at 
$\psi = 0$ and $\psi = \pi/2$ we find
\begin{equation} \label{sminmax}
     s_+^2  + ( 1 - s_+^2) \nu = s_-^2 \,.
\end{equation}

We assume that the somersault axis is the middle principal axis so that $I_2 > I_1 > I_3$.
This means that the somersault is unstable. As a result the dimensionless parameter  $\delta < 0$.

Starting from Corollary~\ref{cor:scale} we find  
\begin{lemma}
The period $\scaledP_3$ of the twisting motion and the change of somersault angle $\Phi_3$
during such a period are given by 
\[
      \gamma  \scaledP_3 =  \frac{8}{(s_+ + s_-) \sqrt{1 -\nu} } K(k^2), \qquad 
      k =  \frac{s_+  -  s_- }{s_+ + s_-}, 
       \quad \nu = \delta / \gamma
\]
and 
\[
      \Phi_3 - \scaledP_3 =  \frac{8}{(s_+ + s_-) \sqrt{1 -\nu} }  \left( \Pi(n_-, k^2) - \Pi( n_+, k^2) \right) , 
      \quad n_\pm = k \frac{1 \pm s_-}{1 \mp s_-}
\]
\end{lemma}
Note that we distinguish the period $\scaledP_3$ from the time  $\scaledT_3$ spent in stage 3;
similarly for the somersault angle $\Phi_3$ per period and somersault angle $\phi_3$ acquired in stage 3.

In this description it is convenient to use both $s_-$ and $s_+$, but of course either one could 
be eliminated with \eqref{sminmax}. Expressing everything in terms
of $s_-$ has the advantage that the limit $s_- \to 0$ corresponds to the approach of the separatrix
of the pure somersault, while $s_- \to 1$ corresponds to the approach of the pure twisting motion.

When the moments of inertia are all distinct, the choice of Euler angles which we
used earlier in the paper is natural because it has the physical interpretation of 
somersault, tilt, and twist. However, this gives a complicated result when the analogue of 
Corollary~\ref{cor:symmetric} is derived. The reason is that in order to solve the energy 
for $\sin\psi$ when $\delta \not = 0$ additional square roots are introduced.
Instead a system of Euler angles needs to be used that has a rotation about the 
axis of the rotor last, say $R = R_1(\tilde\phi) R_3(\tilde\theta) R_2(\tilde\psi)$.
We will use this system of Euler angles to compute the time $T_2$ and the 
somersault $\phi_2$ in stage 2.

\begin{theorem}
An alternate form of the equations of motion for a rigid body with a rotating disc 
attached is given by 
\[
     l \begin{pmatrix} \cos\tilde\theta \cos\tilde\psi \\ -\sin\tilde\theta \\ \cos\tilde\theta \sin\tilde\psi \end{pmatrix} 
     - \begin{pmatrix} 0 \\ h \\ 0 \end{pmatrix}
      = \begin{pmatrix} I_1 & 0 & 0 \\ 0 & I_2 & 0 \\ 0 & 0 & I_3 \end{pmatrix}
      \begin{pmatrix} \cos\tilde\theta \cos\tilde\psi & \sin\tilde\psi & 0 \\ 
                  \sin\tilde\theta & 0 & 1 \\
                  - \cos \tilde\theta \sin \tilde\psi & \cos\tilde\psi & 0
                \end{pmatrix}
      \begin{pmatrix} \dot {\tilde\phi} \\ \dot {\tilde\theta} \\ \dot {\tilde\psi} \end{pmatrix} \,.
\]
The scaled equations of motion are 
\[
  \begin{aligned}
\dot {\tilde \phi} & = 1 + \gamma \sin^2\tilde \psi \\
\dot {\tilde \theta} & =  \gamma \cos\tilde\theta \sin \tilde\psi \cos\tilde\psi\\
\dot {\tilde \psi} & = -(1 + \delta) \rho + \sin\tilde\theta ( \gamma \sin^2\tilde \psi - \delta)
  \end{aligned}
\]
with conserved energy
\[
     E = \frac12 \left( (1 + \gamma \sin^2\tilde\psi)\cos^2\tilde\theta + (1 + \delta)(\rho + \sin^2\tilde\theta)  \right)
\,.
\]
\end{theorem}

In principle the computation of $T_2$ and $\phi_2$ are similar to the symmetric case,
they are just a bit more elaborate. 
For stage 2 only $\theta_{max}$ is defined, because stage 2 always starts with 
$\theta = 0$, but nevertheless we can use $s_-$ from stage 3 as a parameter for stage 2.

\begin{figure}
\includegraphics[width=6cm]{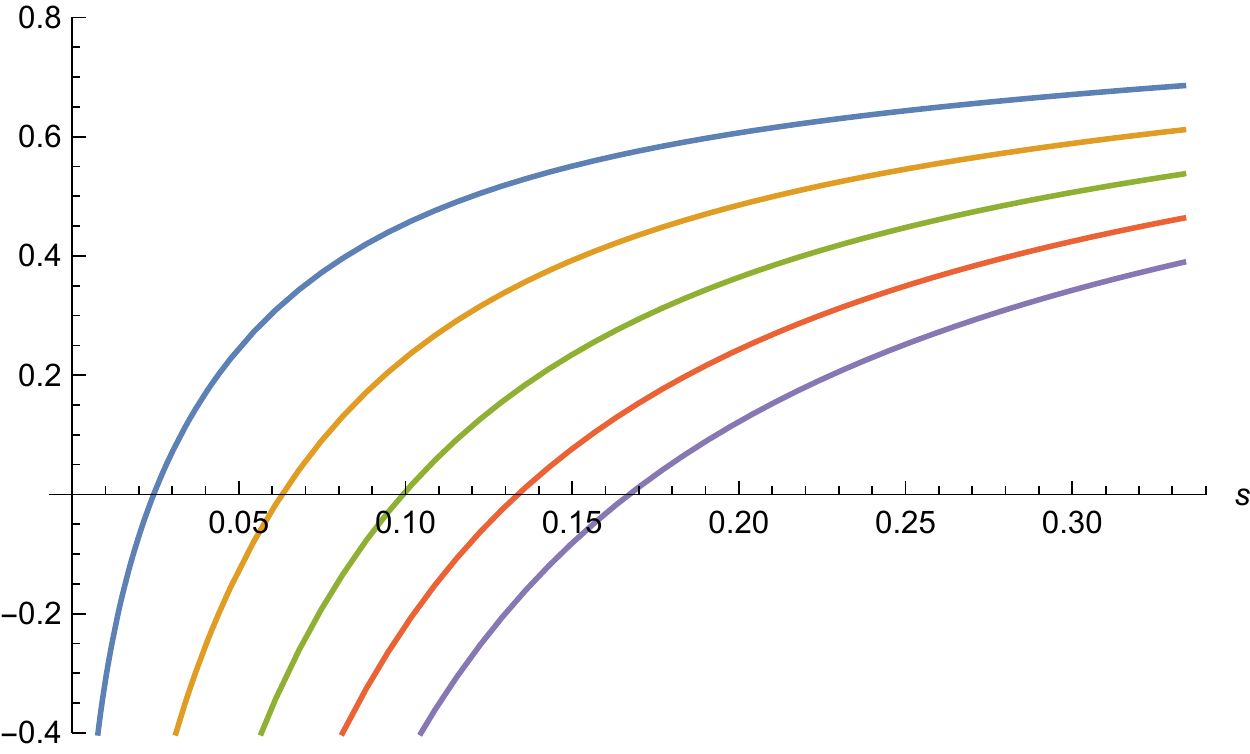}
\includegraphics[width=6cm]{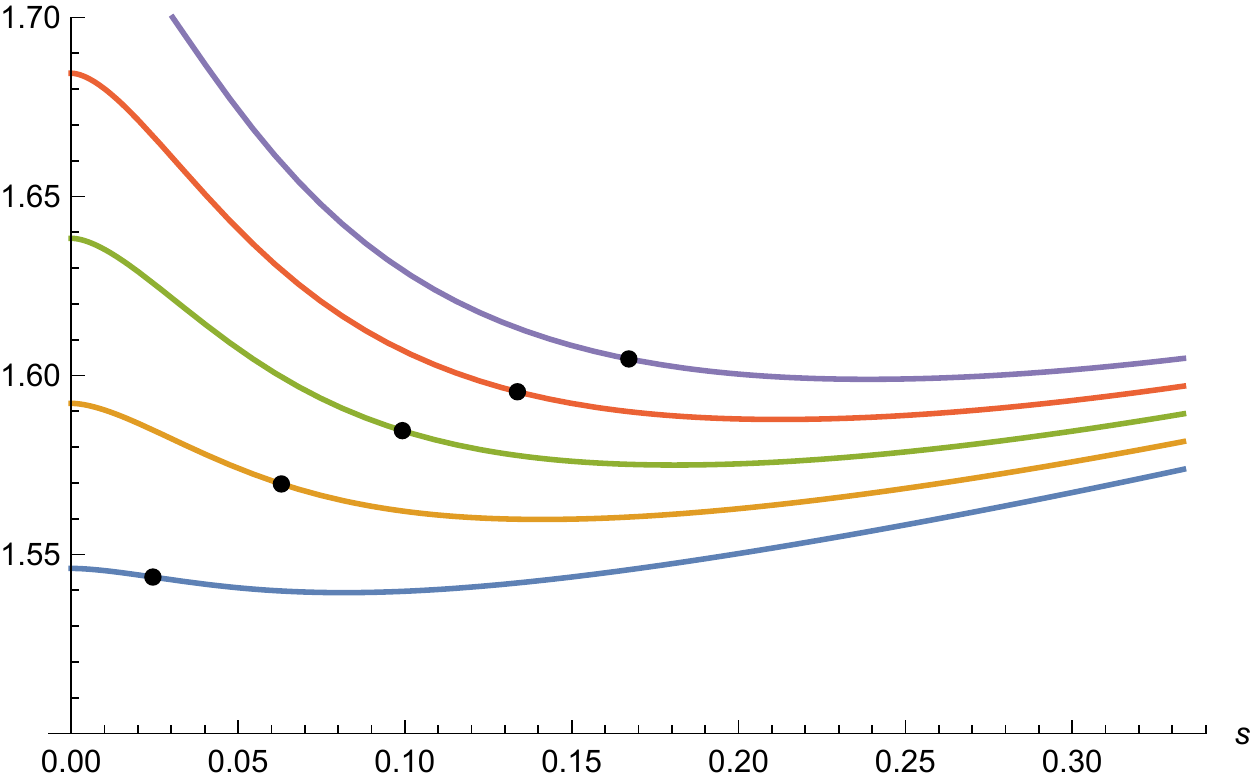}
\includegraphics[width=6cm]{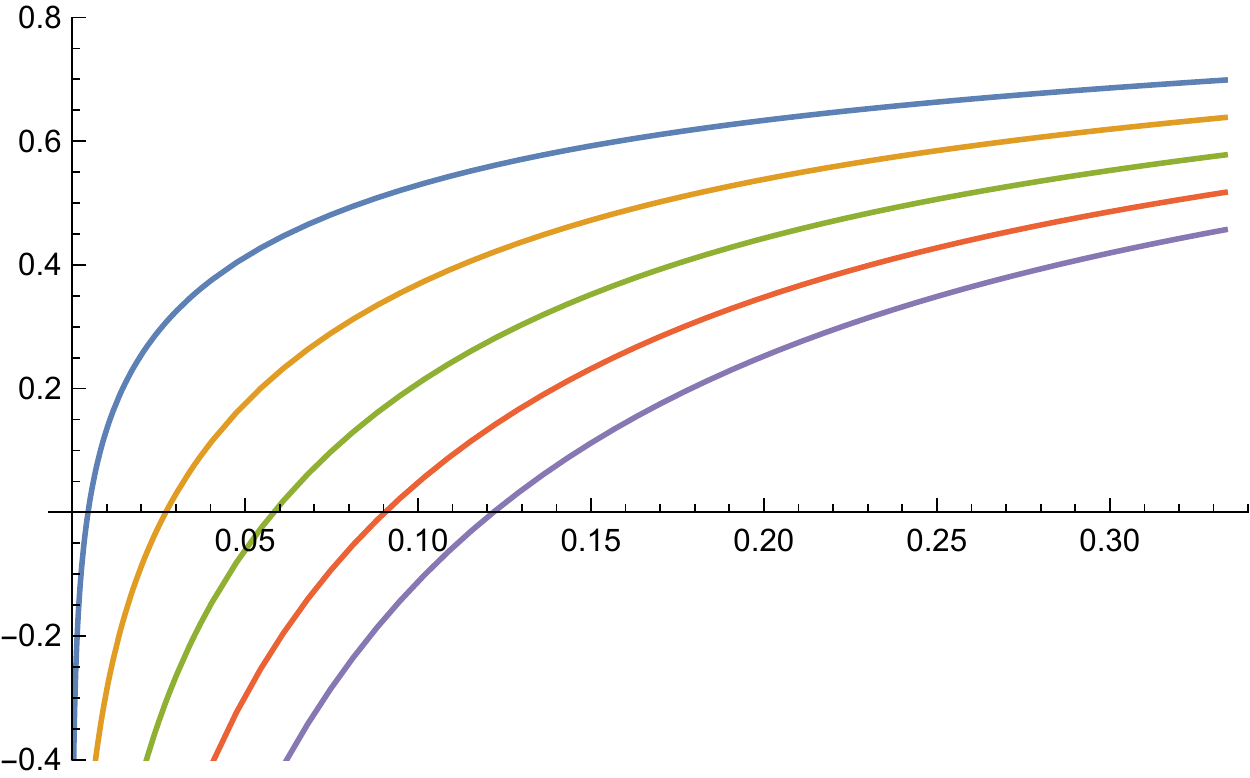}
\includegraphics[width=6cm]{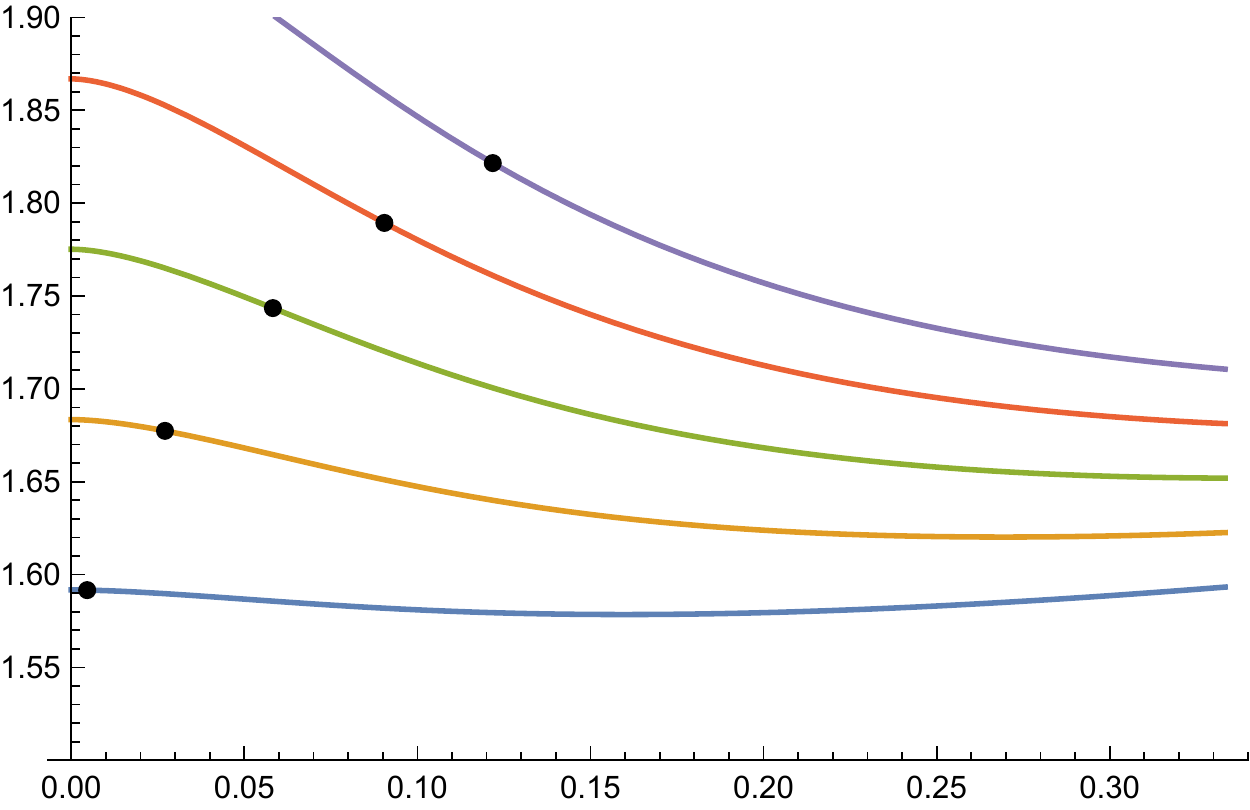}
\caption{Scaled stage one somersaulting time $\hat T_1/(2 \pi)$ (left) and scaled total time $\hat T_{tot}/( 2\pi)$ (right) as a function of $s_- = \sin\theta_{min}$ for $n=1,2,3,4,5$ in the general case for $m=3/2$ somersault, $\gamma = 19$, 
for $\delta = -0.1$ (top row) and $\delta= -0.4$ (bottom row) as given by \eqref{genT1}, \eqref{genTtot}. Dots mark the minimal $s_-$ for which the dive is possible.} \label{fig:asymT}
\end{figure}

\begin{lemma} \label{T2gen}
The  energy with ``rotor on'' is $E_2 = \frac12( 1 + \rho \hat \rho)$
starting at $\L = (l, 0, 0 )$, and the highest point on that 
orbit is $\L = ( 0, -\sin\tilde \theta_{max}, \cos\tilde\theta_{max})$ where
\[
    -\sqrt{ 1 - s_+^2} =-\cos\theta_{max} =  \sin \tilde \theta_{max} = \frac{ \hat \rho  -
                \sqrt{\hat \rho^2 + \gamma( \gamma - \delta )} }{\gamma - \delta} \,.
\]
 The time $T_2$ to move along this orbit segment is given by 
 \[
    \gamma \scaledT_2 =   \frac{1}{k  \sqrt{  s_-^2 ( 1 - \nu)  + \nu  }  } K( k^2) 
 \]
 where
 \[
       k^2 = \frac{(s_-^{-2} -1) ( s_-^2( 1 - \nu) + \nu )}{ (2 - s_-^2)(1 - \nu)}, \quad \nu = \frac{\delta }{\gamma} \,.
 \]
 \end{lemma}
\begin{proof}
The point $\L = (l,0,0)$ corresponds to $\tilde\theta = \tilde\psi = 0$, 
while the point $\L = (0, *,*)$ implies $\tilde\psi = \pi/2$. Evaluating 
the energy at $(l,0,0)$ defines the energy $E_2$ with rotor-on, 
and evaluating $E_2 = E(\tilde\psi = \pi/2)$ defines the endpoint $\tilde \theta_{max}$,
and hence $\theta_{max}$ and hence $s_+$. Then $s_+$ can be expressed in terms 
of $s_-$ using \eqref{sminmax}.

Eliminating $\tilde \psi$ from the ODE for $\dot { \tilde \theta}$, separating 
the variables, and changing variables to $z =  \sin\tilde\theta$ gives
a first kind integral on the elliptic curve 
$ w^2 =  P(z) = z ( z \delta + 2\hat \rho) ( z^2 ( \gamma - \delta)  - 2 z \hat \rho   - \gamma)$.
Using \cite{BF71} this can be written in Legendre's normal form as stated.
\end{proof}

Comparing the (scaled) energy $E_1 = 1/2$ of the pure somersaulting to the 
energy of the twisting $E_3 = E_1 + \frac12 \rho \hat\rho $ we see that 
the change in energy is proportional to $\rho^2$.

The angles $\phi$ and $\tilde \phi$ are related: upon a complete cycle in $\tilde\theta$ and 
$\tilde\psi$ (and hence $\L$) the overall advance in $\tilde \phi$ is the same as the overall 
advance in $\phi$ upon a complete cycle in $\theta$ and $\psi$ (which is the same cycle in $\L$),
but only modulo $2 \pi$. 
It turns out that in our case the total change of the two angles over a period in $\L$ differs by $2\pi$, 
and hence $\phi_2 = \tilde \phi_2 - \pi/2$. 

\begin{lemma}
The change of angle $\tilde \phi_2$ in the general case satisfies
\[
  \tilde \phi_2 - \scaledT_2 =  f_+ \Pi( n_+, k^2) + f_- \Pi( n_-, k^2)
\]
where $k^2$ is as in the previous lemma,  $g = s_- \sqrt{(1 - s_-^2)(2 - s_-^2)(1-\nu)}$, and
\[
    n_\pm = 1 - s_-^{-2} \pm s_-^{-2} \sqrt{ (1 - s_-^2)/( 1 - \nu) }, \quad
    f_\pm = g\left(1 - s_-^2 \mp \sqrt{ (1-s_-^2)(1-\nu) } \right) \,.
 \]
\end{lemma}
\begin{proof}
The change in angle is given by a complete elliptic integral on the same
elliptic curve $w^2 = P(z)$ as in Lemma~\ref{T2gen}, but with an additional rational 
function $R(z)$ obtained from  $\frac{d \tilde \phi }{d \tau} = \frac{ d \tilde \phi }{d  \tilde \theta}  \frac{ d\tilde \theta }{ d\tau }$ and hence
\[
     \tilde \phi_2 = \int R(z) \frac{dz}{w}, \quad 
     R(z) =    1 + \delta 
     			- \frac{ \hat \rho  - \delta/2  }{1+ z}
                       	- \frac{ \hat \rho + \delta /2  }{1 - z} \,.
\]
Using \cite{BF71} this can be written in terms of $K$ and $\Pi$ as stated.
\end{proof}

The conditions to perform a successful dive are as before, namely $ T_{tot} = 2 T_1 + 2 T_2 + P_3(n-1/2)$,
$2m\pi = 2 \phi_1 + 2 \phi_2 + \Phi_3(n-1/2)$ and 
$2 n \pi = 2 \psi_1 + 2 \psi_2 + 2\pi ( n - 1/2)$. 
The last equation is trivially satisfied with $\psi_1 = 0$, $\psi_2 = \pi/2$.
This leads to our final result:

\begin{theorem}
To perform a dive with $m$ somersaults and $n$ twists in a total time $T_{tot}$ 
\begin{equation} \label{genTtot}
    \scaledT_{tot} - 2 m \pi  = 2 ( \scaledT_2 - \phi_2) + ( \scaledP_3 - \Phi_3) ( n - \tfrac12 ) 
\end{equation}
and
\begin{equation}\label{genT1}
     2 \scaledT_1 = 2 m \pi - 2 \phi_2 - \Phi_3 ( n - \tfrac12) \ge 0 
\end{equation}
need to be satisfied. 
\end{theorem}

Note that unlike in the symmetric case there is now a dependence of $T_{tot}$ on $n$.
What remains the same in the general case is that the right hand side of the equation for $T_{tot}$ 
depends on the essential parameter $s_-$ and the asymmetry parameter $\nu$ only. Thus as before
$s_-$ can be adjusted after takeoff to achieve the desired dive. 
The right hand side of the equation for $T_1$ in addition depends on $\gamma$,
as it did in the symmetric case.

It turns out that the series expansion in the limit of small $s_{-}$
has a rather limited radius of convergence, and so it is not useful to describe 
the range of $s_{-}$ of interest to us. 
The reason behind this difficulty is that for $s_{-} \to 0$ the  
elliptic integrals have a logarithmic divergence.

Figure~\ref{fig:asymT} shows how the dives change when the asymmetry increases. 
The main observations are that with increasing asymmetry the necessary minimal 
tilt decreases, while the total necessary time increases. This can be understood by the 
presence of a separatrix in the asymmetric case. On the one hand the motion along 
the separatrix (even without the rotor) will produce an increase in tilt, in fact 
from $s_-$ to $s_+$, as introduced above. On the other hand the motion near the 
separatrix will take longer when near $s_-$. In fact motion along the separatrix with
rotor off takes an infinite amount of time, and this is the reason why the series 
expansions are not useful in this limit.

\section{Geometric Phase}

The change in the somersault angle $\phi$ during a completion of a loop in $\L$ can be split into two contributions,
a so called dynamic and a geometric phase. We start discussing this in the simplest case of symmetric moments of inertia $I_1 = I_2$ 
and with rotor off, $\rho = 0$. In this case we can write 
\[
    \Phi_3 = \frac{ 2 E_3 P_3}{l} - S_3
\]
where $P_3$ is the period of the twisting motion, $E_3$ the corresponding energy, and $S_3$ is the solid angle enclosed by the 
trajectory on the $\L$ sphere. For $I_1=I_2$ the trajectory in the $\L$ sphere is a circle $\theta = const$, and the solid angle enclosed by 
this curve and the equator is $S_3 = 2 \pi \sin \theta$ where $\theta = 0$ is the equator (zero solid angle), and $\theta = \pi/2$ is the pole
(solid angle of half the sphere). 
The period we found before, it is $P_3 = 2\pi I_1 /( l \gamma \sin \theta)$, while $E_3 = l^2 ( 1 + \gamma \sin^2\theta)/(2 I_1)$.
As noted before $\Phi_3 =  l P_3/I_1$, and hence we have verified the above identity. 

For general moments of inertia the above formula still holds, and has first been derived by Montgomery \cite{Montgomery90}.
Of course now the expression for $\Phi_3$, $P_3$ and $S_3$ are all complete elliptic integrals. A slight difference 
to Montgomery is that fixing a particular angle $\phi$ allowed us to remove the $\mod 2\pi$ in the formula.
This is rather important in the present application since we would like to distinguish say between 1/2 and 3/2 somersaults.

\begin{figure}
\includegraphics[width=15cm]{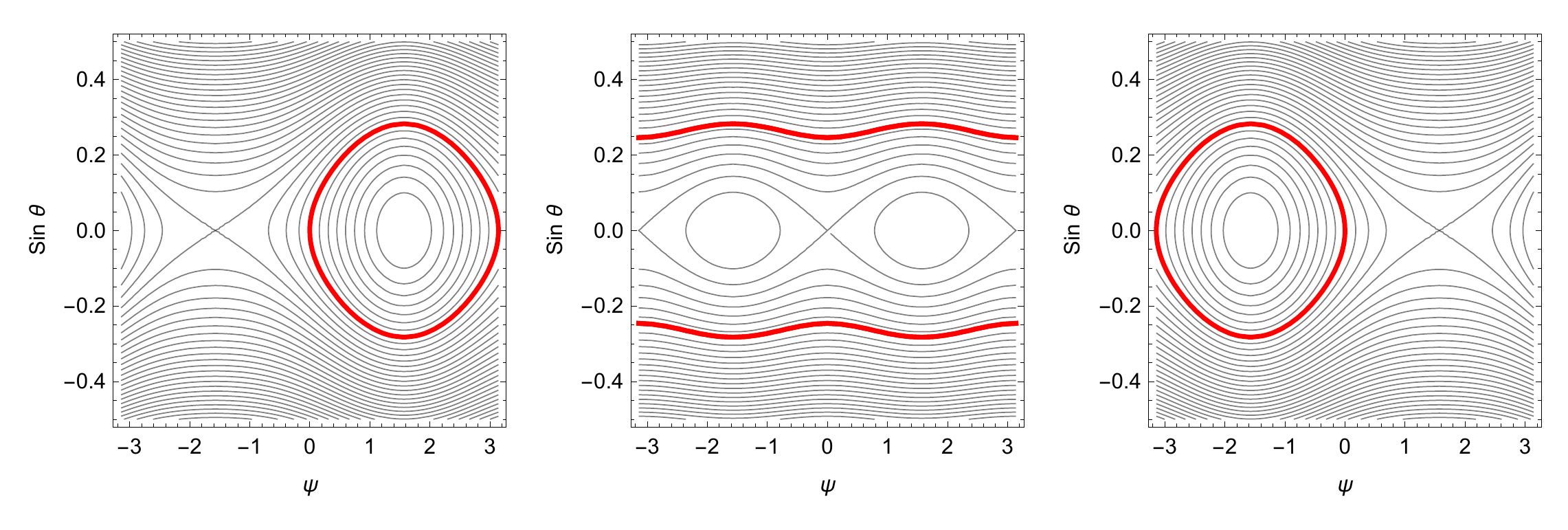}
\caption{Phase portraits in $(\psi, \sin \theta)$ for stage 2, 3, 4, corresponding to 
``rotor left", ``rotor off", ``rotor right". Parts of the three pieces of trajectories indicated in bold joined 
together give a dive as computed by the formulas in the previous section. 
Parameters are $\gamma = 19$, $\delta = -0.4$, $\rho = 1$. } \label{fig:phase234}
\end{figure}

Including the ``rotor on" stage, but returning to the symmetric case we now have a closed loop in $\L$ that consists of 
pieces from stages 2, 3, and 4. 
There is a generalisation of Montgomery's formula due to Cabrera \cite{Cabrera07}, which allows for a general shape change. 
Before we write down the corresponding expression consider the three phase portraits for stage 2, 3, 4 on the $\L$ sphere.
Here we present the $\L$-sphere in spherical coordinates $(\psi, \theta)$. This coordinate system is singular for 
$\theta = \pm \pi/2$, but the motions we are interested in stay away from this point. In Figure~\ref{fig:phase234} 
 three phase portraits are shown with part of a trajectory indicated as a thick line, which are the ones that 
 are used to construct the closed loop.
 
In Cabrera's formula the geometric phase is still given by the solid angle of the area enclosed. 
What changes is the dynamics phase where instead of $2 E T$ we have  
$\int \L \cdot \O \,\dee t$. With the initial choice of Euler angles we have 
\[
       \L \cdot \O = l ( \dot \phi + \dot \psi \sin\theta) \,.
\]
The solid angle integral can be thought of as $\oint p dq$ where here 
$q =\psi$ and $p = \sin \theta$, so that the second term will be cancelled by the solid angle $S$ (with appropriate 
sign and orientation), and the remaining integral gives the change in $\phi$, as desired. 
When $h = 0$ (rotor off) then it is easy to see that $\L \cdot \O = 2 E$ and Montogomery's formula
is recovered. 
Again, for our particular choice of angle we measure the area relative to the equator, which allows 
us to remove the $\mod 2\pi$.
The beauty of the resulting formulas is that they give us a good intuition of what the change in the somersault angle is going 
to be without actually computing it. Instead we just need to know the corresponding times and energies which give the 
dynamic phase, and we need to know the area enclosed by the equator and the curve on the $\L$ sphere.
Nothing will change in the description of the dive as presented in the previous sections, it is merely the 
interpretation of the answer that changes.

\section{Acknowledgement}

This research was partially supported by ARC Linkage grant LP100200245 and the New South Wales Institute of Sports.

\bibliographystyle{plain}
\bibliography{../bib_cv/all,../bib_cv/hd,../bib_cv/SoSa}

\end{document}